\documentclass[a4paper,draft]{amsart}
\usepackage{amsmath, amsthm, amssymb}
\usepackage{url}
\usepackage{braket}

\theoremstyle{plain}
\newtheorem{thm}{Theorem}[section]
\newtheorem{theorem}[thm]{Theorem}

\newtheorem{lemma}[thm]{Lemma}
\newtheorem{cor}[thm]{Corollary}

\theoremstyle{remark}  
\newtheorem{remark}[thm]{Remark}

\newtheorem*{ackn}{Acknowledgments}

\theoremstyle{definition}  

\newcommand{\aaa}{\boldsymbol{a}}
\newcommand{\bbb}{\boldsymbol{b}}

\newcommand{\xxx}{\boldsymbol{x}}
\newcommand{\yyy}{\boldsymbol{y}}

\newcommand{\KK}{\mathbb{K}}
\newcommand{\NN}{\mathbb{N}}

\newcommand{\BBB}{\mathcal{B}}

\begin{document}

\title[On the determinant of multiplication map]{On the determinant of multiplication map of a monomial complete intersection ring}

\author[Y. Numata]{Yasuhide NUMATA}
\thanks{The author was partially supported by JSPS KAKENHI Grant Number JP18K03206.}
\address{Department of Mathematics, Shinshu University, Matsumoto, Japan.}
\email{nu@math.shinshu-u.ac.jp}
\begin{abstract}
In this article, we consider the monomial complete intersection algebra $\mathbb{K}[x,y]/\langle x^d,y^q\rangle$ in two variables.
For elements $l_1,\ldots,l_{d+q-2k}$ of degree $1$,
we give a formula of the deteminant of linear map
  from the homogeneous component of degree $k$ to the homogenous component of degree $d+q-k$ defined by the multiplication of $l_1 \cdots l_{d+q-2k}$.
\end{abstract}

\maketitle

\section{Introduction}
Roughly speaking,
the strong Lefschetz property for an algebra
is an analogue of the property for
the cohomology ring of compact K\"ahler variety
known as Hard Lefschetz Theorem
(see also \cite{MR3112920, MR3161738}).
We say that
a graded Artinian Gorenstein algebra $R=\bigoplus_{k=0}^{s} R_k$
with socle degree $s$
has
the strong Lefschetz property with a Lefschetz element $l\in R_1$
if
the multiplication map
$\times l^{s-2k}\colon R_k \ni f \mapsto fl^{s-2k} \in R_{s-k} $
is bijective for any $0\leq k\leq s/2$.
For some algebras with the strong Lefschetz property,
characterization of Lefschetz elements is studied,
e.g. \cite{MR3789915,MR2817446}.
The determinant of the linear map
$\times l^{s-2k}\colon R_k\to  R_{s-k} $
is also studied, e.g. \cite{MR2459403, 1812.07199, 1904.01800}.

Consider the monomial complete intersection ring
\begin{align*}
R=\KK[x_1,\ldots,x_n]/\braket{x_1^{d_1},\ldots,x_1^{d_n}}.
\end{align*}
Then the algebra $R$ has the strong Lefschetz property
(see also \cite{MR1136655, MR578321, MR951211}).
For example,
$l=x_1+\cdots+x_n$ is a Lefschetz element for the algebra.
In \cite{MR2459403},
the determinant of
multiplication map of $l^{s-2k}$
is calculated
in the case where $d_1=\cdots=d_n=1$.
In the case where $n=2$,
the determinants are implicitly given in \cite{MR1701596}.
In this article,
we generalize the problem to the multiplication map of
a product of $l_1,\ldots,l_{s-2k}\in R_1$.
We consider
the determinant of the linear map
$\times l_1\cdots l_{s-2k}\colon R_k\ni f\mapsto f\cdot ( l_1\cdots l_{s-2k}) \in R_{s-2k} $
for $l_1,\ldots,l_{s-2k}\in R_1$.
In the case where $n=2$,
the determinant can be written with Schur polynomials.

This article is organized as follows:
In Section \ref{sec:notation},
we recall notation and facts of symmetric polynomials and determinants.
In Section \ref{sec:rep},
we calculate the determinant of the multiplication map.

\begin{ackn}
The author thanks anonymous referees for valuable suggestions.
\end{ackn}

\section{Notation and Formulae}
\label{sec:notation}
In this section,
 we recall notation and facts which will be used in Section \ref{sec:rep}.

We call  
a squence $\lambda=(\lambda_1,\lambda_2,\ldots)$ of
weakly decreasing nonnegative intergers
\emph{a partiotion of $m$}
if $\sum_i \lambda_i = m$.
For nonnegative integers $r,l\in\NN$,
$(r^l)$ denotes
the partition 
consisting of $l$ copies of $r$. 
For a parition $\lambda$ of $n$,
define $\widetilde\lambda_j=\Set{i| 1 \leq j \leq \lambda_i}$.
Then $\widetilde\lambda=(\widetilde\lambda_1,\widetilde\lambda_2,\ldots)$
is also a parition of $n$.
We call $\widetilde\lambda$ the \emph{conjugate parition} to $\lambda$.
For example,
the partition $(l^r)$ is the conjugate parition to $(r^l)$.
For partitions $\mu$ and $\lambda$,
we write $\mu\subset \lambda$
to denote that
they satisfy
$\mu_i \leq \lambda_i$ for all $i$.
For $\lambda\subset (r^l)$,
we define $(r^l)\setminus \lambda$
to be the parition $(r-\lambda_{l},r-\lambda_{l-1},\ldots,r-\lambda_{1},0,\ldots)$.


For a partition $\lambda$,
we define the \emph{Schur polynomial} $s_\lambda(\xxx)$ in 
$n$ variables $\xxx=(x_1,\ldots,x_n)$
to be
\begin{align*}
s_\lambda(\xxx)=
\frac
{\det\left((x_j)^{\lambda_i + n -i}\right )_{\substack{i=1,\ldots,n\\j=1,\ldots,n}}} 
{\det\left((x_j)^{n -i}\right )_{\substack{i=1,\ldots,n\\j=1,\ldots,n}}} .
\end{align*}
For a partition $\lambda$ of a nonnegative integer $m$,
the Schur polynomial  $s_\lambda(\xxx)$ is
a homogeneous symmetric polynomial of degree $m$.
For $k$, the Schur polynomial $s_{(1^k)}(\xxx)$
is the $k$-th \emph{elementary symmetric polynomial} $e_k(\xxx)$,
i.e.,
the sum 
\begin{align*}
\sum_{1\leq i_1<i_2<\cdots <i_k \leq n} x_{i_1} x_{i_2}\cdots  x_{i_k}
\end{align*}
of all square-free monomials of 
degree $k$.
If $k>n$ or $k< 0$, then $e_k(\xxx)=0$.
If $k=0$, then $e_0(\xxx)=1$.
It is known that
the Schur polynomial and elementary symmetric polynomial satisfies
\begin{align}
s_{\widetilde\lambda}(\xxx)=\det
\left(e_{\lambda_i + j-i}(\xxx)\right)_{\substack{i=1,\ldots,l\\j=1,\ldots,l}},
\label{eq:JT}
\end{align}
for a parition such that $\lambda_{l+1}=0$.
(See e.g. \cite[Section I.3]{MR1354144}.)

Next we recall the Cauchy--Binet formula
for determinants.
Assume that $m\geq n$.
Let 
$X$ be an $n\times m$ matrix,
$Y$ an $m\times n$ matrix.
For $S\subset\Set{1,2,\ldots,n}$ with $\#S=m$,
$X^S$ (\emph{resp.}\ $Y_S$) denotes the $m\times m$ submatrix
whose rows (\emph{resp.}\ columns) are the rows  (\emph{resp.}\ columns) of $X$
(\emph{resp.}\ $Y$) at indices from $S$.
If the entries of $X$ and $Y$ are elements of a commutative ring,
then we have the equation
\begin{align}
 \det(YX)=
\sum_{S} \det(Y_S)\det(X^S),
\label{eq:CB}
\end{align}
where the sum is over all
subsets  $S\subset\Set{1,2,\ldots,n}$ such that $\#S=m$.
(See e.g. \cite[Section 5.6]{MR2289254}.)

\section{The determinant of representation matrices.}
\label{sec:rep}
Let $\KK$ be a field.
For positive integers $d,q$ with $d\geq q$,
we consider the algebra $R=\KK[x,y]/\braket{x^{d+1},y^{q+1}}$.
The algebra $R$ can be decompose into homogeneous spaces $R_k$
as follows: $R=\bigoplus_{k=0}^{d+q} R_k$.
The set $\BBB_k=\Set{x^{i}y^{k-i}| 0\leq i \leq d,\ 0\leq k-i \leq q}$
is a $\KK$-basis for the homogeneous space $R_k$.
Hence we have
\begin{align*}
  \dim_{\KK} R_k =
  \begin{cases}
    k+1 & (0\leq k\leq q)\\
    q+1 & (q\leq k \leq d)\\
    s-k +1 & (q \leq k \leq d+q).
  \end{cases}
\end{align*}
Let $k\leq \frac{d+q}{2}$.
Then $\dim_{\KK} R_k = \dim_{\KK} R_{d+q-k}$.
Let
$l_1=a_1x+b_1y,l_2=a_2x+b_2y,\ldots \in R_1\setminus\Set{0}$.
For $l_t$, we define a linear map
\begin{align*}
  \times l_t\colon R_{k+t-1} \ni f \mapsto f\cdot l_t \in R_{k+t}.
\end{align*}
Then we obtain the following sequence of linear maps:
\begin{align*}
  R_k\xrightarrow{\times l_1}
  R_{k+1}\xrightarrow{\times l_{2}}
  \cdots
  \xrightarrow{\times l_{d+q-2k}}
  R_{d+q-k}.
\end{align*}
First we calculate the representation matrix of the linear map
\begin{align*}
  \times l_1\cdot l_2\cdots l_{u}\colon R_{k}\ni f \mapsto f \cdot  l_1\cdot l_2\cdots l_{u} \in R_{k+u}.
\end{align*}
with respect to the bases $\BBB_{k}$ and $\BBB_{k+u}$.
\begin{lemma}
  Assume that $\beta=b_1b_2\cdots b_u\neq 0$.
  Let $l_1=a_1x+b_1y,l_2=a_2x+b_2y,\ldots,l_u=a_ux+b_uy$.
  Then the coefficient of 
  $x^{w+u-i}y^{i}$
  in
  $x^{w-j}y^{j}\cdot l_1 l_2\cdots l_u$
  is
 $\beta e_{u+j-i} \left(\frac{\aaa}{\bbb}\right)$,
where  
$\frac{\aaa}{\bbb}=\left(\frac{a_1}{b_1},\frac{a_2}{b_2},\ldots,\frac{a_u}{b_u}\right)$.
\end{lemma}
\begin{proof}
  Since $b_i \neq 0$ for all $i$, we have
  \begin{align*}
    l_1 l_2\cdots l_u
    &=(a_1 x+ b_1y)(a_2 x+ b_2y)\cdots (a_u x+ b_uy)\\
    &=\beta \cdot\left(\frac{a_1}{b_1} x+ y\right)\left(\frac{a_2}{b_2} x+ y\right)\cdots \left(\frac{a_u}{b_u} x+ y\right)\\
    &=\beta\cdot \sum_{i} e_{u-i}
\left(\frac{\aaa}{\bbb}\right)
    x^{u-i}y^{i}.
  \end{align*}
  Hence
  \begin{align*}
    x^{w-j}y^{j}\cdot l_1 l_2\cdots l_u
    &=
\beta\cdot
\sum_{i} e_{u-i}
\left(\frac{\aaa}{\bbb}\right)
    x^{w-j+u-i}y^{i+j}\\
    &=
\beta\cdot \sum_{i} 
e_{u+j-i}
\left(\frac{\aaa}{\bbb}\right)
    x^{w+u-i}y^{i}.
  \end{align*}
\end{proof}
 
Now we calculate the determinant $D_{d,q}(a_1,\ldots,a_{d+q-2k};b_1,\ldots,b_{d+q-2k})$
of the linear map
\begin{align*}
  \times l_1\cdot l_2\cdots l_{d+q-2k}\colon R_{k}\ni f \mapsto f \cdot  l_1\cdot l_2\cdots l_{d+q-2k} \in R_{d+q-k}.
\end{align*}
\begin{theorem}[Main Theorem]
  \label{thm:repInGeneral}
  Let $l_1=a_1x+b_1y,l_2=a_2x+b_2y,\ldots ,l_{d+q-2k}=a_{d+q-2k}x+b_{d+q-2k}y \in R_1\setminus\Set{0}$.
Let $i$ be a permutation on $\Set{1,2,\ldots,d+q-2k}$.
Assume that 
 $\check \beta=  b_{i_1}b_{i_2}\cdots b_{i_u} \neq 0$, and that
 $\hat \alpha=  a_{i_{u+1}}a_{i_{u+1}}\cdots a_{i_{d+q-2k}}\neq 0$.
Let
\begin{align*}
\aaa&=(a_1,\ldots,a_{d+q-2k}),
&
\bbb&=(b_1,\ldots,b_{d+q-2k}),\\
 \frac{\check{\aaa}}{\check{\bbb}} &=
  \left(
  \frac{a_{i_1}}{b_{i_1}},
  \frac{a_{i_2}}{b_{i_2}},
  \ldots,
  \frac{a_{i_u}}{b_{i_u}}
  \right) ,
&
 \frac{\hat{\bbb}}{\hat{\aaa}} &=
  \left(
  \frac{b_{i_{u+1}}}{a_{i_{u+1}}},
  \frac{b_{i_{u+2}}}{a_{i_{u+2}}},\ldots,
  \frac{b_{i_{d+q-2k}}}{a_{i_{d+q-2k}}}
  \right).
\end{align*}
    
If $q\leq k\leq\frac{q+d}{2}$, then
\begin{align*}
D_{d,q}(\aaa;\bbb)
=
\hat \alpha^{q+1} \check \beta^{q+1}
  \cdot
  s_{((q+1)^u)}
  \left(
  \frac{\check\aaa}{\check\bbb}
  \right)
  s_{((q+1)^{d+q-2k-u})}
  \left(
  \frac{\hat\bbb}{\hat\aaa}
  \right).
\end{align*}

If $0\leq k\leq k+u \leq q$,
then 
\begin{align*}
D_{d,q}(\aaa;\bbb) &=
\hat \alpha^{k+1} \check \beta^{k+1}
  \cdot
  \sum_{\lambda\subset(u^{k+1})} 
  s_{\widetilde\lambda}
  \left(
  \frac{\check\aaa}{\check\bbb}
  \right)
  s_{\widetilde{(d^{k+1})\setminus\lambda}}
  \left(
  \frac{\hat\bbb}{\hat\aaa}
  \right).
  \end{align*}

If $0\leq k \leq q \leq d \leq k+u$, then
\begin{align*}
D_{d,q}(\aaa;\bbb) &=
\hat \alpha^{k+1} \check \beta^{k+1}
  \cdot
  \sum_{\lambda\subset(u^{k+1})} 
  s_{\widetilde{(d^{k+1})\setminus\lambda}}
  \left(
  \frac{\check\aaa}{\check\bbb}
  \right)
  s_{\widetilde\lambda}
  \left(
  \frac{\hat\bbb}{\hat\aaa}
  \right).
  \end{align*}

If $k\leq q\leq k+u \leq d$, then
\begin{align*}
D_{d,q}(\aaa;\bbb)=
\hat \alpha^{k+1} \check \beta^{k+1}
  \cdot
  \sum_{\lambda\subset((q-k)^{k+1})} 
  s_{\widetilde\lambda}
  \left(
  \frac{\check\aaa}{\check\bbb}
  \right)
  s_{\widetilde{(d^{k+1})\setminus\lambda}}
  \left(
  \frac{\hat\bbb}{\hat\aaa}
  \right).
  \end{align*}
\end{theorem}
\begin{proof}
Since $R$ is a commutative algebra,
we can assume 
$i_t=t$ without loss of generality.
In this case,
\begin{align*}
 \frac{\check{\aaa}}{\check{\bbb}} &=
  \left(
  \frac{a_{1}}{b_{1}},
  \frac{a_{2}}{b_{2}},
  \ldots,
  \frac{a_{u}}{b_{u}}
  \right) ,
&
 \frac{\hat{\bbb}}{\hat{\aaa}} &=
  \left(
  \frac{b_{u+1}}{a_{u+1}},
  \frac{b_{u+2}}{a_{u+2}},\ldots,
  \frac{b_{d+q-2k}}{a_{d+q-2k}}
  \right),
\\
 \check \beta &= b_1b_2\cdots b_u,
&\hat \alpha&=a_{u+1}a_{u+2}\cdots a_{d+q-2k}.
\end{align*}

In the case where
$q\leq k\leq\frac{q+d}{2}$,
 the bases $\BBB_k$, $\BBB_{k+u}$ and $\BBB_{d+q-k}$ are 
  \begin{gather*}
    \Set{x^{k}y^{0},x^{k-1}y^{1},\ldots, x^{k-q}y^{q}},\\
    \Set{x^{k+u}y^{0},x^{k+u-1}y^{1},\ldots, x^{k+u-q}y^{q}},\\
    \Set{x^{d+q-k}y^{0},x^{d+q-k-1}y^{1},\ldots, x^{d-k}y^{q}},
  \end{gather*}
  respectively.
  Hence the representation matrix $X$
  for $\times l_1\cdots l_u\colon R_k \to R_{k+u}$
  is
  \begin{align*}
    \left(
    \check \beta  
    e_{u+j-i}
    \left(\frac{\check\aaa}{\check\bbb}\right)
    \right)_{\substack{i=0,1,\ldots,q\\ j=0,1,\ldots,q}},
  \end{align*}
  and the representation matrix $Y$
  for  $\times l_{u+1}\cdots l_{d+q-2k}\colon R_{k+u} \to R_{d+q-k}$
  is
  \begin{align*}
    \left(
    \hat \alpha 
    e_{d+q-2k-u+j-i}
    \left(\frac{\hat\bbb}{\hat\aaa}\right)
    \right)_{\substack{i=0,1,\ldots,1\\ j=0,1,\ldots,q}}.
  \end{align*}
  Hence
\begin{align*}
  D_{d,q}(\aaa;\bbb)
  =&\det(Y)\det(X) \\
  =&
  \check \beta^{q+1}
s_{((q+1)^{u})}
  \left(\frac{\check\aaa}{\check\bbb}\right)
  \cdot
  \hat \alpha^{q+1}
s_{((q+1)^{d+q-2k-u})}
  \left(\frac{\hat\bbb}{\hat\aaa}\right).
\end{align*}

Next we consider the case where $0\leq k\leq k+u \leq q$.
In this case,
the bases $\BBB_k$, $\BBB_{k+u}$ and $\BBB_{d+q-k}$ are 
\begin{gather*}
    \Set{x^{k}y^{0},x^{k-1}y^{1},\ldots, x^{0}y^{k}},\\
    \Set{x^{k+u}y^{0},x^{k+u-1}y^{1},\ldots, x^{0}y^{k+u}},\\
    \Set{x^{d}y^{q-k},x^{d-1}y^{q-k+1},\ldots, x^{d-k}y^{q}},
\end{gather*}
respectively.
  Hence the representation matrix $X$
  for $\times l_1\cdots l_u\colon R_k \to R_{k+u}$
  is
  \begin{align*}
    \left(
    \check \beta 
    e_{u+j-i}
    \left(\frac{\check\aaa}{\check\bbb}\right)
    \right)_{\substack{i=0,1,\ldots,k+u\\ j=0,1,\ldots,k}},
  \end{align*}
  and the representation matrix $Y$
  for  $\times l_{u+1}\cdots l_{d+q-2k}\colon R_{k+u} \to R_{d+q-k}$
  is
  \begin{align*}
    \left(
    \hat \alpha 
    e_{d+q-2k-u+j-i}
    \left(\frac{\hat\bbb}{\hat\aaa}\right)
    \right)_{\substack{i=q-k,q-k+1,\ldots,q\\ j=0,1,\ldots,k+u}}
&=
    \left(
    \hat \alpha 
    e_{d-k-u+k-i}
    \left(\frac{\hat\bbb}{\hat\aaa}\right)
    \right)_{\substack{i=0,1,\ldots,k\\ j=0,1,\ldots,k+u}}.
  \end{align*}
To calculate $\det(YX)$, 
we consider minors $\det(X^{\Set{\delta_0,\ldots,\delta_k}})$ and
 $\det(Y_{\Set{\delta_0,\ldots,\delta_k}})$ for
$0\leq \delta_0<\delta_1<\cdots < \delta_k\leq k+u$.
Let $\lambda_{i}=u-\delta_i+i$,
and $\mu_{k-j}=d-u+\delta_j-j$.
Since $0\leq \delta_0<\delta_1<\cdots < \delta_k\leq k+u$,
it follows that
$\lambda=(\lambda_0,\lambda_1,\ldots,\lambda_k,0,\ldots)$
and 
$\mu=(\mu_0,\mu_1,\ldots,\mu_k,0,\ldots)$
are paritions satisfying 
$\lambda\subset (u^{k+1})$ and $\mu=(d^{k+1})\setminus\lambda$.
The minor $\det(X^{\Set{\delta_0,\ldots,\delta_k}})$
 is equal to
  \begin{align*}
\det \left(
    \check \beta
    e_{u+j-i}
    \left(\frac{\check\aaa}{\check\bbb}\right)
    \right)_{\substack{i=\delta_0,\delta_1,\ldots,\delta_k\\   j=0,1,\ldots,k}}
    &=    
    \check \beta^{k+1} 
    \det\left(
    e_{u+j-\delta_i}
    \left(\frac{\check\aaa}{\check\bbb}\right)
    \right)_{\substack{i=0,1,\ldots,k\\ j=0,1,\ldots,k}}\\
    &=    
    \check \beta^{k+1} 
    \det\left(
    e_{\lambda_i+j-i}
    \left(\frac{\check\aaa}{\check\bbb}\right)
    \right)_{\substack{i=0,1,\ldots,k\\ j=0,1,\ldots,k}}.
  \end{align*}
  It follows from \eqref{eq:JT} that 
  \begin{align*}
   \det(X^{\Set{\delta_0,\ldots,\delta_k}})
=
    \check \beta^{k+1} 
    s_{\widetilde\lambda}
    \left(\frac{\check\aaa}{\check\bbb}\right).
  \end{align*}
  On the other hand,
 the minor $Y_{\Set{\delta_0,\ldots,\delta_k}}$  is
  \begin{align*}
    \det \left(
    \hat \alpha 
   e_{d-k-u+k-i}
    \left(\frac{\hat\bbb}{\hat\aaa}\right)
    \right)_{\substack{i=0,1,\ldots,k\\   j=\delta_0,\delta_1,\ldots,\delta_k}}
&   =    
    \hat \alpha^{k+1} 
    \det \left(
    e_{d-u-k+\delta_j-i}
    \left(\frac{\hat\bbb}{\hat\aaa}\right)
    \right)_{\substack{i=0,1,\ldots,k\\ j=0,1,\ldots,k}}\\
&   =    
    \hat \alpha^{k+1} 
    \det \left(
    e_{\mu_{k-j}+j-i}
    \left(\frac{\hat\bbb}{\hat\aaa}\right)
    \right)_{\substack{i=0,1,\ldots,k\\ j=0,1,\ldots,k}}.
  \end{align*}
  By flipping vertically and horizontally,
  we have 
  \begin{align*}
   \det(Y_{\Set{\delta_0,\ldots,\delta_k}})
   &   =    
    \hat \alpha^{k+1} 
    \det \left(
    e_{\mu_{j}+i-j}
    \left(\frac{\hat\bbb}{\hat\aaa}\right)
    \right)_{\substack{i=0,1,\ldots,k\\ j=0,1,\ldots,k}}.
  \end{align*}
  It follows from \eqref{eq:JT} formula that
  \begin{align*}
   \det(Y_{\Set{\delta_0,\ldots,\delta_k}})
   &   =    
    \hat \alpha^{k+1} 
    s_{\widetilde\mu}
    \left(\frac{\hat\bbb}{\hat\aaa}\right).
  \end{align*}
  Therefore it follows from \eqref{eq:CB} formula
  that
  \begin{align*}
    \det(YX)
    &=
   \sum_{0\leq \delta_0<\delta_1<\cdots < \delta_k\leq k+u} \det(Y_{\Set{\delta_0,\ldots,\delta_j}})\det(X^{\Set{\delta_0,\ldots,\delta_j}})\\
    &=
    (\hat\alpha\check\beta)^{k+1} 
    \sum_{\lambda\subset(u^{k+1})} 
    s_{\widetilde\mu}
    \left(\frac{\hat\bbb}{\hat\aaa}\right)
    s_{\widetilde\lambda}
    \left(\frac{\check\aaa}{\check\bbb}\right),   
  \end{align*}
  where 
  $\mu=(d^{k+1})\setminus \lambda$.

Next we consider the case where $0\leq k\leq q$ and $d \leq k+u$.
Since
$D_{d,q}(\aaa;\bbb)=D_{q,d}(\bbb;\aaa)$,
we can obtain the formula in this case
from the result for the case where $0\leq k\leq k+u \leq q$.

Finally we consider the case where $k\leq q\leq k+u \leq d$.
In this case,
the bases $\BBB_k$, $\BBB_{k+u}$ and $\BBB_{d+q-k}$ are 
\begin{gather*}
    \Set{x^{k}y^{0},x^{k-1}y^{1},\ldots, x^{0}y^{k}},\\
    \Set{x^{k+u}y^{0},x^{k+u-1}y^{1},\ldots, x^{k+u-q}y^{q}},\\
    \Set{x^{d}y^{q-k},x^{d-1}y^{q-k+1},\ldots, x^{d-k}y^{q}},
\end{gather*}
respectively.
  Hence the representation matrix $X$
  for $\times l_1\cdots l_u\colon R_k \to R_{k+u}$
  is
  \begin{align*}
    \left(
    \check \beta 
    e_{u+j-i}
    \left(\frac{\check\aaa}{\check\bbb}\right)
    \right)_{\substack{i=0,1,\ldots,q\\ j=0,1,\ldots,k}},
  \end{align*}
  and the representation matrix $Y$
  for  $\times l_{u+1}\cdots l_{d+q-2k}\colon R_{k+u} \to R_{d+q-k}$
  is
  \begin{align*}
    \left(
    \hat \alpha 
    e_{d+q-2k-u+j-i}
    \left(\frac{\hat\bbb}{\hat\aaa}\right)
    \right)_{\substack{i=q-k,q-k+1,\ldots,q\\ j=0,1,\ldots,q}}
&=
    \left(
    \hat \alpha 
    e_{d-k-u+j-i}
    \left(\frac{\hat\bbb}{\hat\aaa}\right)
    \right)_{\substack{i=0,1,\ldots,k\\ j=0,1,\ldots,q}}.
  \end{align*}

  Hence we obtain a formula similar to the case where $0\leq k\leq k+u \leq q$.
  For $0\leq \delta_0<\delta_1<\cdots < \delta_k\leq q$,
  we obtain a partition $\lambda$ by  $\lambda_{i}=u-\delta_{i}+i$.
  In this case, $\lambda$ is a partition contained by $((q-k)^{k+1})$.
  Hence we obtain the formula.
\end{proof}

As corollary to Theorem \ref{thm:repInGeneral},
we have an explicit formula for the determinant in a special case.
\begin{cor}
  \label{thm:rep}
  Assume that $\beta=b_1b_2\cdots b_{d+q-2k}\neq 0$.
 Let
\begin{align*}
 \aaa&=(a_1,\ldots,a_{d+q-2k}),
&
 \bbb&=(b_1,\ldots,b_{d+q-2k}),
&
 \frac{\aaa}{\bbb}&=\left(\frac{a_1}{b_1},\frac{a_2}{b_2},\ldots,\frac{a_{d+q-2k}}{b_{d+q-2k}}\right).
\end{align*}
If $0\leq k \leq q$, then
\begin{align*}
 D_{d,q}(\aaa;\bbb)=
     \beta^{k+1}
s_{((k+1)^{d-k})}
  \left(\frac{\aaa}{\bbb}\right).
\end{align*}
If $q\leq k \leq \frac{d+q}{2}$,
then
\begin{align*}
 D_{d,q}(\aaa;\bbb)=
    \beta^{q+1}
s_{((q+1)^{d+q-2k})}
  \left(\frac{\aaa}{\bbb}\right). 
\end{align*}
\end{cor}
In the case where $d=q$,
the defining ideal of the ring is symmetric.
Hence we have the following formula:
\begin{cor}
  \label{cor:schur}
Assume 
that $\alpha=a_1a_2 \cdots a_{2m}\neq 0$
and that $\beta=b_1b_2 \cdots b_{2m}\neq 0$.
Let 
\begin{align*}
\left(\frac{\aaa}{\bbb}\right)
&=
 \left(\frac{a_1}{b_1},\frac{a_2}{b_2},\ldots,\frac{a_{2m}}{b_{2m}}\right),
&
\left(\frac{\bbb}{\aaa}\right)
&=  \left(\frac{b_1}{a_1},\frac{b_2}{a_2},\ldots,\frac{b_{2m}}{a_{2m}}\right) .
\end{align*}
Then the following equation holds:
\begin{align*}
  \beta^{r}
s_{(r^{m})}
\left(\frac{\aaa}{\bbb}\right)
  =
  \alpha^{r}
 s_{(r^{m})}
\left(\frac{\bbb}{\aaa}\right).
\end{align*}
\end{cor}
\begin{proof}
  Consider the case where $d=q$.
  By Corollary \ref{thm:rep},
  the determinant 
  for $a_1x+b_1y,\ldots,a_{2d-2k}x+b_{2d-2k}y$ is
  \begin{align*}
    D_{d,d}(\aaa;\bbb)=
    \beta^{k+1}
s_{((k+1)^{d-k})}
    \left(\frac{\aaa}{\bbb}\right).
  \end{align*}
  On the other hand,
  the determinant 
  for $b_1x+a_1y,\ldots,b_{2d-2k}x+a_{2d-2k}y$ is
  \begin{align*}
    D_{d,d}(\bbb;\aaa)=
    \alpha^{k+1}
s_{((k+1)^{d-k})}
    \left(\frac{\bbb}{\aaa}\right).
  \end{align*}
  Since the defining ideal $\braket{x^{d+1},y^{d+1}}$ of the ring $R$ is symmetric,
  these determinants are the same.
  Hence we have
  \begin{align*}
    \beta^{k+1}
s_{((k+1)^{d-k})}
    \left(\frac{\aaa}{\bbb}\right)
    =
    \alpha^{k+1}
s_{((k+1)^{d-k})}
    \left(\frac{\bbb}{\aaa}\right).
  \end{align*}
  Let $2m=2d-2k$ and $r=k+1$. Then we have the equation.
\end{proof}
\begin{remark}
For  $\lambda\subset (r^n)$ and $\mu=(r^n)\setminus \lambda$,
the Schur polynomials satisfy the following equation 
(see \cite[Exercise 7.41]{MR1676282}):
  \begin{align}
    (y_1y_2 \cdots y_{n})^{r}
    \cdot s_{\lambda}
    \left(\frac{\xxx}{\yyy}\right)
    =
    (x_1x_2 \cdots x_{n})^{r}
    \cdot s_{\mu}
    \left(\frac{\yyy}{\xxx}\right),
\label{eq:EC2}
  \end{align}
where $\frac{\xxx}{\yyy}=\left(\frac{x_1}{y_1},\frac{x_2}{y_2},\ldots,\frac{x_{n}}{y_{n}}\right)$, 
and
$\frac{\yyy}{\xxx}=\left(\frac{y_1}{x_1},\frac{y_2}{x_2},\ldots,\frac{y_{n}}{x_{n}}\right)$.
  Let $\lambda=\mu=(r^m)$ and $n=2m$.
  Then $\mu=(r^n)\setminus \lambda$.
  In this case,
  Equation \eqref{eq:EC2} is
  the equation in Corollary \ref{cor:schur}.
\end{remark}

\bibliography{by-mr}

\providecommand{\bysame}{\leavevmode\hbox to3em{\hrulefill}\thinspace}
\providecommand{\MR}{\relax\ifhmode\unskip\space\fi MR }
\providecommand{\MRhref}[2]{%
  \href{http://www.ams.org/mathscinet-getitem?mr=#1}{#2}
}
\providecommand{\href}[2]{#2}
\begin{thebibliography}{10}

\bibitem{MR3789915}
Mats Boij, Juan Migliore, Rosa~M. Mir\'{o}-Roig, and Uwe Nagel, \emph{The
  non-{L}efschetz locus}, J. Algebra \textbf{505} (2018), 288--320, URL
  \url{https://doi.org/10.1016/j.jalgebra.2018.03.006}. \MR{3789915}

\bibitem{MR2289254}
Harry Dym, \emph{Linear algebra in action}, Graduate Studies in Mathematics,
  vol.~78, American Mathematical Society, Providence, RI, 2007. \MR{2289254}

\bibitem{MR2459403}
Masao Hara and Junzo Watanabe, \emph{The determinants of certain matrices
  arising from the {B}oolean lattice}, Discrete Math. \textbf{308} (2008),
  no.~23, 5815--5822, URL \url{https://doi.org/10.1016/j.disc.2007.09.055}.
  \MR{2459403}

\bibitem{MR3112920}
Tadahito Harima, Toshiaki Maeno, Hideaki Morita, Yasuhide Numata, Akihito
  Wachi, and Junzo Watanabe, \emph{The {L}efschetz properties}, Lecture Notes
  in Mathematics, vol. 2080, Springer, Heidelberg, 2013, URL
  \url{https://doi.org/10.1007/978-3-642-38206-2}. \MR{3112920}

\bibitem{MR1701596}
C.~Krattenthaler, \emph{Advanced determinant calculus}, vol.~42, 1999, The
  Andrews Festschrift (Maratea, 1998), pp.~Art. B42q, 67. \MR{1701596}

\bibitem{MR1354144}
I.~G. Macdonald, \emph{Symmetric functions and {H}all polynomials}, second ed.,
  Oxford Mathematical Monographs, The Clarendon Press, Oxford University Press,
  New York, 1995, With contributions by A. Zelevinsky, Oxford Science
  Publications. \MR{1354144}

\bibitem{MR2817446}
Toshiaki Maeno, Yasuhide Numata, and Akihito Wachi, \emph{Strong {L}efschetz
  elements of the coinvariant rings of finite {C}oxeter groups}, Algebr.
  Represent. Theory \textbf{14} (2011), no.~4, 625--638, URL
  \url{https://doi.org/10.1007/s10468-010-9207-9}. \MR{2817446}

\bibitem{MR3161738}
Juan Migliore and Uwe Nagel, \emph{Survey article: a tour of the weak and
  strong {L}efschetz properties}, J. Commut. Algebra \textbf{5} (2013), no.~3,
  329--358, URL \url{https://doi.org/10.1216/JCA-2013-5-3-329}. \MR{3161738}

\bibitem{1904.01800}
Takahiro Nagaoka and Akiko Yazawa, \emph{Strict log-concavity of the kirchhoff
  polynomial and its applications to the strong lefschetz property}, 2019,
  Eprint \url{arXiv:1904.01800}.

\bibitem{MR1136655}
Les Reid, Leslie~G. Roberts, and Moshe Roitman, \emph{On complete intersections
  and their {H}ilbert functions}, Canad. Math. Bull. \textbf{34} (1991), no.~4,
  525--535, URL \url{https://doi.org/10.4153/CMB-1991-083-9}. \MR{1136655}

\bibitem{MR578321}
Richard~P. Stanley, \emph{Weyl groups, the hard {L}efschetz theorem, and the
  {S}perner property}, SIAM J. Algebraic Discrete Methods \textbf{1} (1980),
  no.~2, 168--184, URL \url{https://doi.org/10.1137/0601021}. \MR{578321}

\bibitem{MR1676282}
\bysame, \emph{Enumerative combinatorics. {V}ol. 2}, Cambridge Studies in
  Advanced Mathematics, vol.~62, Cambridge University Press, Cambridge, 1999,
  With a foreword by Gian-Carlo Rota and appendix 1 by Sergey Fomin, URL
  \url{https://doi.org/10.1017/CBO9780511609589}. \MR{1676282}

\bibitem{MR951211}
Junzo Watanabe, \emph{The {D}ilworth number of {A}rtinian rings and finite
  posets with rank function}, Commutative algebra and combinatorics ({K}yoto,
  1985), Adv. Stud. Pure Math., vol.~11, North-Holland, Amsterdam, 1987,
  pp.~303--312, URL \url{https://doi.org/10.2969/aspm/01110303}. \MR{951211}

\bibitem{1812.07199}
Akiko Yazawa, \emph{The hessians of the complete and complete bipartite graphs
  and its application to the strong lefschetz property}, 2018, Eprint
  \url{arXiv:1812.07199}.

\end{thebibliography}
\bibliographystyle{amsplain-url} 
\end{document}